\documentclass[a4paper,12pt]{article}
\usepackage[left=2.50cm, right=2.50cm, top=1cm, bottom=2.50cm]{geometry} 
\usepackage{amsfonts,amsmath,latexsym,amssymb,amsthm,mathrsfs,upref} 
\usepackage{enumitem}        
\usepackage{caption}

\newtheorem{thm}{Theorem}

\newtheorem{pro}{Proposition}             
\newtheorem{cor}{Corollary}

\newcommand\keywords[1]{\quad\quad\, \textbf{Keywords}: #1}

\title{Two iterative formulas of largest and smallest singular value of nonsingular matrices}%
\author{\small {Shun Xu}\\
	{\small School of Mathematical Sciences, Tongji University, Shanghai, 200092, China}\\
	{\small  e-mail: shunxu1997@163.com}
}%
\date{}%

\begin{document}
	\maketitle
	\begin{abstract}
			We obtain an iterative formula that converges incrementally to the smallest singular value. Similarly, we obtain an iterative formula that converges decreasingly to the largest singular value.
	\end{abstract}
	\keywords{Singular values, Frobenius norm, Determinant.}
	
\section{Lower bound for the smallest singular value}
Let $M_{n}(\mathbb{C})(n\geqslant2)$ be the space of $n \times n$ complex matrices. Let $\sigma_{i}$ $(i=1, \cdots, n)$ be the singular values of $A \in M_{n}(\mathbb{C})$ which is nonsingular and suppose that $\sigma_{1} \geqslant \sigma_{2} \geqslant \cdots \geqslant \sigma_{n-1} \geqslant \sigma_{n} > 0$. For
$A=\left[a_{i j}\right] \in M_{n}(\mathbb{C})$, the Frobenius norm of $A$ is defined by
$$
\|A\|_{F}=\left(\sum_{i, j=1}^{n}\left|a_{i j}\right|^{2}\right)^{1 / 2}={\rm tr}\left( A^{H} A\right)^{\frac{1}{2}}
$$
where $A^{H}$ is the conjugate transpose of $A$. The relationship between the Frobenius norm and singular values is
$$
\|A\|_{F}^{2}=\sigma_{1}^{2}+\sigma_{2}^{2}+\cdots+\sigma_{n}^{2}
$$
It is well known that lower bounds for the smallest singular value $\sigma_{n}$ of a nonsingular matrix $A \in M_{n}(\mathbb{C})$ have many potential theoretical and practical applications \cite{1985Matrix,horn1994topics}. Yu and Gu \cite{yi1997note} obtained a lower bound for $\sigma_{n}$ as follows:
$$
\sigma_{n} \geqslant |\det A|\left(\frac{n-1}{\|A\|_{F}^{2}}\right)^{(n-1) / 2}=l>0
$$
The above inequality is also shown in \cite{piazza2002upper}.  In \cite{zou2012lower}, Zou improved the above inequality by showing that
$$
\sigma_{n } \geqslant|{\det} A|\left(\frac{n-1}{\|A\|_{F}^{2}-l^2}\right)^{(n-1) / 2}=l_0
$$
In \cite{lin2021some}, Lin and Xie improve a lower bound for
smallest singular value of matrices by showing that $a$ is the smallest positive solution to the equation
$$
x^{2}\left(\|A\|_{F}^{2}-x^{2}\right)^{n-1}=|{\det} A|^{2}(n-1)^{n-1} .
$$
and $\sigma_{n} \geqslant a>l_{0}$. Under certain conditions, $\sigma_{n}=a$ will hold. However, in many cases, $\sigma_{n}=a$ is not true. We give necessary and sufficient conditions such that $\sigma_{n}=a$ in Proposition \ref{pro1} .

\section{Main results} 
Let
$$
\begin{gathered}
	l=|\operatorname{det} A|\left(\frac{n-1}{\|A\|_{F}^{2}}\right)^{(n-1) / 2}, \\
	l_{0}=|\operatorname{det} A|\left(\frac{n-1}{\|A\|_{F}^{2}-l^{2}}\right)^{(n-1) / 2} .
\end{gathered}
$$
And $a$ is the smallest positive solution to the equation
$$
x^{2}\left(\|A\|_{F}^{2}-x^{2}\right)^{n-1}=|\operatorname{det} A|^{2}(n-1)^{n-1} .
$$
From \cite{lin2021some}, we have $\sigma_{n} \geqslant a>l_{0}>l>0$. Next, we give necessary and sufficient conditions such that $a=\sigma_{n}$.
\begin{pro}\label{pro1}
	Let a be the smallest positive solution to the equation
	$$
	x^{2}\left(\|A\|_{F}^{2}-x^{2}\right)^{n-1}=|\operatorname{det} A|^{2}(n-1)^{n-1} .
	$$
	then $a=\sigma_{n}$ if and only if
	$$
	\sigma_{1}=\sigma_{2}=\cdots=\sigma_{n-1}
	$$
\end{pro}
\begin{proof}
	$\Rightarrow:$  If $a=\sigma_{n}$, then
	$$
	\sigma_{n}^{2}\left(\|A\|_{F}^{2}-\sigma_{n}^{2}\right)^{n-1}=|\operatorname{det} A|^{2}(n-1)^{n-1} .
	$$
	Since $\|A\|_{F}^{2}=\sigma_{1}^{2}+\sigma_{2}^{2}+\cdots+\sigma_{n}^{2}$ and $|\operatorname{det} A|^{2}=\sigma_{1}^{2} \sigma_{2}^{2} \cdots \sigma_{n}^{2}$, we have
	$$
	\begin{gathered}
		\sigma_{n}^{2}\left(\sigma_{1}^{2}+\sigma_{2}^{2}+\cdots+\sigma_{n-1}^{2}\right)^{n-1}=\sigma_{1}^{2} \sigma_{2}^{2} \cdots \sigma_{n}^{2}(n-1)^{n-1} . \\
		\left(\frac{\sigma_{1}^{2}+\sigma_{2}^{2}+\cdots+\sigma_{n-1}^{2}}{n-1}\right)^{n-1}=\sigma_{1}^{2} \sigma_{2}^{2} \cdots \sigma_{n-1}^{2}
	\end{gathered}
	$$
	By the arithmetic-geometric mean inequality, the above equation holds if and only if
	$$
	\sigma_{1}^{2}=\sigma_{2}^{2}=\cdots=\sigma_{n-1}^{2}
	$$
	if and only if
	$$
	\sigma_{1}=\sigma_{2}=\cdots=\sigma_{n-1}
	$$
	
	$\Leftarrow:$ Let
	$$
	\sigma_{1}^{2}=\sigma_{2}^{2}=\cdots=\sigma_{n-1}^{2}=c^{2}(c>0) .
	$$
	Then $a$ is the smallest positive solution to the equation
	$$
	x^{2}\left(c^{2}(n-1)+\sigma_{n}^{2}-x^{2}\right)^{n-1}=c^{2(n-1)} \sigma_{n}^{2}(n-1)^{n-1}
	$$
	if and only if $a$ is the smallest positive zero point of
	$$
	f(x)= x^2\left(1+\frac{\sigma_{n}^{2}-x^{2}}{c^{2}(n-1)}\right)^{n-1}-\sigma_{n}^{2} .
	$$
	Next, we proof $a=\sigma_{n}$. Obviously, we can see that $f\left(\sigma_{n}\right)=0$ and $f(0)=-\sigma_{n}^{2}<0$. Next, we prove that $f(x)$ is an strictly increasing function on $\left[0, \sigma_{n}\right]$. Taking the derivative of $f(x)$, we can get
	$$
	f^{\prime}(x)=\left(1+\frac{\sigma_{n}^{2}-x^{2}}{c^{2}(n-1)}\right)^{n-2} \frac{2 x}{c^{2}(n-1)}\left(c^{2}(n-1)+\sigma_{n}^{2}-n x^{2}\right) .
	$$
	$\forall x_{0} \in\left(0, \sigma_{n}\right)$, we have
	$$
	c^{2}(n-1)+\sigma_{n}^{2}-n x_{0}^{2} \geqslant n \sigma_{n}^{2}-n x_{0}^{2}>0 .
	$$
	We have $f^{\prime}\left(x_{0}\right)>0$. Therefore $f(x)$ is an strictly increasing function on $\left[0, \sigma_{n}\right]$ and $\sigma_{n}$ is the smallest positive zero point of
	$$
	f(x)= x^2\left(1+\frac{\sigma_{n}^{2}-x^{2}}{c^{2}(n-1)}\right)^{n-1}-\sigma_{n}^{2}
	$$
	Therefore, $\sigma_{n}$ is the smallest positive solution to the equation
	$$
	x^{2}\left(c^{2}(n-1)+\sigma_{n}^{2}-x^{2}\right)^{n-1}=c^{2(n-1)} \sigma_{n}^{2}(n-1)^{n-1}.
	$$
	We get $a=\sigma_{n}.$
\end{proof}
In the above special condition, $a$ can be equal to $\sigma_{n}$, which is not general. Next, we give our main theorem. We give an iterative formula for the smallest singular value, which converges incrementally to the smallest singular value.
\begin{thm}\label{thm1}
	Let $A \in M_{n}(\mathbb{C})$ be nonsingular and $0<a_{1} \leqslant \sigma_{n}$ and
	$$
	a_{k+1}=\left(a_{k}^{2}+\left|\operatorname{det}\left(a_{k}^{2} I_{n}-A^{H} A\right)\right|\left(\frac{n-1}{\|A\|_{F}^{2}-(n-1) a_{k}^{2}}\right)^{n-1}\right)^{1 / 2}, k=1,2, \cdots .
	$$
	Then $\sigma_{n} \geqslant a_{k+1} \geqslant a_{k}>0(k=1,2, \cdots)$ and $\lim _{k \rightarrow+\infty} a_{k}=\sigma_{n}$.
\end{thm}
\begin{proof}
	Let $0 \leqslant \lambda<\sigma_{n}^{2}$, by the arithmetic-geometric mean inequality, we have
	$$
	\left(\sigma_{1}^{2}-\lambda\right)\left(\sigma_{2}^{2}-\lambda\right) \cdots\left(\sigma_{n-1}^{2}-\lambda\right) \leqslant\left(\frac{\sigma_{1}^{2}+\cdots+\sigma_{n-1}^{2}-(n-1) \lambda}{n-1}\right)^{n-1}
	$$
	Since
	$$
	\begin{aligned}
		\left(\sigma_{1}^{2}-\lambda\right)\left(\sigma_{2}^{2}-\lambda\right) \cdots\left(\sigma_{n-1}^{2}-\lambda\right) &=\frac{\left(\sigma_{1}^{2}-\lambda\right)\left(\sigma_{2}^{2}-\lambda\right) \cdots\left(\sigma_{n}^{2}-\lambda\right)}{\sigma_{n}^{2}-\lambda} \\
		&=\frac{\left|\operatorname{det}\left(\lambda I_{n}-A^{H} A\right)\right|}{\sigma_{n}^{2}-\lambda}
	\end{aligned}
	$$
	we have
	$$
	\frac{\left|\operatorname{det}\left(\lambda I_{n}-A^{H} A\right)\right|}{\sigma_{n}^{2}-\lambda} \leqslant\left(\frac{\sigma_{1}^{2}+\cdots+\sigma_{n-1}^{2}-(n-1) \lambda}{n-1}\right)^{n-1}
	$$
	$$
	\begin{gathered}
		\sigma_{n}^{2} \geqslant \lambda+\left|\operatorname{det}\left(\lambda I_{n}-A^{H} A\right)\right|\left(\frac{n-1}{\sigma_{1}^{2}+\cdots+\sigma_{n-1}^{2}-(n-1) \lambda}\right)^{n-1} \\
		\sigma_{n} \geqslant\left(\lambda+\left|\operatorname{det}\left(\lambda I_{n}-A^{H} A\right)\right|\left(\frac{n-1}{\sigma_{1}^{2}+\cdots+\sigma_{n-1}^{2}+\sigma_{n}^{2}-(n-1) \lambda}\right)^{n-1}\right)^{1 / 2} \\
		\sigma_{n} \geqslant\left(\lambda+\left|\operatorname{det}\left(\lambda I_{n}-A^{H} A\right)\right|\left(\frac{n-1}{\|A\|_{F}^{2}-(n-1) \lambda}\right)^{n-1}\right)^{1 / 2}
	\end{gathered}
	$$
	Let $\lambda \rightarrow \sigma_{n}^{2-}\left(\lambda\right.$ tends to $\sigma_{n}^{2}$ from the left). We get that the above inequality is also true for $\lambda=\sigma_{n}^{2}$. Therefore, for $0 \leqslant \lambda \leqslant \sigma_{n}^{2}$, we have
	\begin{equation}\label{eq1}
		\sigma_{n} \geqslant\left(\lambda+\left|\operatorname{det}\left(\lambda I_{n}-A^{H} A\right)\right|\left(\frac{n-1}{\|A\|_{F}^{2}-(n-1) \lambda}\right)^{n-1}\right)^{1 / 2}
	\end{equation}
	We show by induction on $k$ that
	$$
	\sigma_{n} \geqslant a_{k+1} \geqslant a_{k}>0
	$$
	We have $\sigma_{n} \geqslant a_{1}>0$. In equation \ref{eq1}, let $\lambda=a_{1}^{2}$, we have
	$$
	\sigma_{n} \geqslant a_{2}=\left(a_{1}^{2}+\left|\operatorname{det}\left(a_{1}^{2} I_{n}-A^{H} A\right)\right|\left(\frac{n-1}{\|A\|_{F}^{2}-(n-1) a_{1}^{2}}\right)^{n-1}\right)^{1 / 2} \geqslant a_{1}>0
	$$
	Assume that our claim is true for $k=m$, that is $\sigma_{n} \geqslant a_{m+1} \geqslant a_{m}>0$. Now we consider the case when $k=m+1$. In equation \ref{eq1}, let $\lambda=a_{m+1}^{2}$, we have
	$$
	\sigma_{n} \geqslant a_{m+2}=\left(a_{m+1}^{2}+\left|\operatorname{det}\left(a_{m+1}^{2} I_{n}-A^{H} A\right)\right|\left(\frac{n-1}{\|A\|_{F}^{2}-(n-1) a_{m+1}^{2}}\right)^{n-1}\right)^{1 / 2} \geqslant a_{m+1}>0
	$$
	Hence $\sigma_{n} \geqslant a_{m+2} \geqslant a_{m+1}>0$. This proves $\sigma_{n} \geqslant a_{k+1} \geqslant a_{k}>0(k=1,2, \cdots)$. By the well known monotone convergence theorem, $\lim _{k \rightarrow \infty} a_{k}$ exists. Let $\lim _{k \rightarrow \infty} a_{k}=$ $\sigma$, then
	$$
	\sigma=\left(\sigma^{2}+\left|\operatorname{det}\left(\sigma^{2} I_{n}-A^{H} A\right)\right|\left(\frac{n-1}{\|A\|_{F}^{2}-(n-1) \sigma^{2}}\right)^{n-1}\right)^{1 / 2}, k=1,2, \cdots
	$$
	We have
	$$
	\left|\operatorname{det}\left(\sigma^{2} I_{n}-A^{H} A\right)\right|\left(\frac{n-1}{\|A\|_{F}^{2}-(n-1) \sigma^{2}}\right)^{n-1}=0
	$$
	Since
	$$
	\frac{n-1}{\|A\|_{F}^{2}-(n-1) \sigma^{2}} \neq 0
	$$
	we have
	$$
	\operatorname{det}\left(\sigma^{2} I_{n}-A^{H} A\right)=0 .
	$$
	We get that $\sigma^{2}$ is the eigenvalue of $A^{H} A$. Since $\sigma_{n}^{2}$ is the smallest eigenvalue of $A^{H} A$, we have $\sigma^{2} \geqslant \sigma_{n}^{2}$. According to the definition of $\sigma$, we have $\sigma \leqslant \sigma_{n}$. Therefore, $\sigma^{2} \leqslant \sigma_{n}^{2}$ and we get $\sigma=\sigma_{n}$. Hence $\lim _{k \rightarrow+\infty} a_{k}=\sigma_{n}$.
\end{proof}

From Theorem \ref{thm1}, we can see that as long as there is a lower bound of $\sigma_{n}$ (we set it to $b$ ) and let $a_{1}=b$ in Theorem \ref{thm1} , we can get a better lower bound than $b$. For example, we bring the lower bound of Lin and Xie \cite{lin2021some} into our Theorem \ref{thm1} to obtain the following results.
\begin{cor}\label{cor1}
	Let a be the smallest positive solution to the equation
	$$
	x^{2}\left(\|A\|_{F}^{2}-x^{2}\right)^{n-1}=|\operatorname{det} A|^{2}(n-1)^{n-1} .
	$$
	Let $a_{1}=a$ and
	$$
	a_{k+1}=\left(a_{k}^{2}+\left|\operatorname{det}\left(a_{k}^{2} I_{n}-A^{H} A\right)\right|\left(\frac{n-1}{\|A\|_{F}^{2}-(n-1) a_{k}^{2}}\right)^{n-1}\right)^{1 / 2}, k=1,2, \cdots
	$$
	Then $\sigma_{n} \geqslant a_{k+1} \geqslant a_{k} \geqslant a>0(k=1,2, \cdots)$ and $\lim _{k \rightarrow+\infty} a_{k}=\sigma_{n}$.
\end{cor}
\begin{proof}
	From \cite{lin2021some}, we have $\sigma_{n} \geqslant a>0$. We know Corollary \ref{cor1} is a special case of Theorem \ref{thm1}.
\end{proof} 

We give the  iterative formula for the smallest singular value:
\[
a_{k+1}=\left(a_k^2+|\det(a_k^2 I_n-A^HA)|\left(\frac{n-1}{\|A\|_F^2-(n -1) a_k^2}\right)^{n-1}\right)^{1/2}, k=1,2,\cdots,
\]
where $a_1=a$, which converges incrementally to the smallest singular value.
Similarly, we can give an iterative formula that converges decreasingly to the largest singular value.

\begin{thm}\label{thm2}
	Let $A \in M_{n}(\mathbb{C})$ be nonsingular, $a_1\geqslant\sigma_1$. Assume
	\[
	a_{k+1}=\left(a_k^2-|\det(a_k^2 I_n-A^HA)|\left(\frac{n-1}{(n+1) a_k^2-\| A\|_F^2}\right)^{n-1}\right)^{1/2}, k=1,2,\cdots.
	\]
	Then $\sigma_1\leqslant a_{k+1}\leqslant a_k(k=1,2,\cdots)$ and $\lim_{k\to +\infty}a_k=\sigma_1$.
\end{thm}

\begin{proof}
	Set $\lambda>\sigma_1^2$, according to the arithmetic geometric mean inequality, we can get
	\[
	\left(\lambda-\sigma_{2}^{2}\right)\left(\lambda-\sigma_{3}^{2}\right) \cdots\left(\lambda-\sigma_{n}^ 2\right) \leqslant\left(\frac{(n-1) \lambda-(\sigma_2^2 +\cdots+\sigma_{n}^{2})}{n-1}\right)^{n -1}.
	\]
	Since
	$$
	\begin{aligned}
		\left(\lambda-\sigma_{2}^{2}\right)\left(\lambda-\sigma_{3}^{2}\right) \cdots\left(\lambda-\sigma_{n}^ 2\right)
		&=\frac{\left(\lambda-\sigma_{1}^{2}\right)\left(\lambda-\sigma_{2}^{2}\right) \cdots\left(\lambda-\sigma_{n}^2\right)}{\lambda-\sigma_1^2}\\
		&=\frac{|\det(\lambda I_n-A^HA)|}{\lambda-\sigma_1^2}
	\end{aligned},
	$$
	we can get
	\[
	\frac{|\det(\lambda I_n-A^HA)|}{\lambda-\sigma_1^2}\leqslant\left(\frac{(n-1) \lambda-(\sigma_2^2 +\cdots+ \sigma_{n}^{2})}{n-1}\right)^{n-1}
	\]
	$$
	\begin{aligned}
		\sigma_{1}^{2} &\leqslant \lambda-|\operatorname{det}\left( \lambda-A^{H} A\right) \mid\left(\frac{n-1}{( n-1) \lambda+\sigma_{1}^{2}-\|A\|_{F}^{2}}\right)^{n-1} \\
		& \leqslant \lambda-\left|\operatorname{det}\left(\lambda-A^{H} A\right)\right|\left(\frac{n-1}{(n+1) \lambda -\|A\|_{F}^{2}}\right)^{n-1}
	\end{aligned}
	$$
	$$
	\sigma_1\leqslant\left(\lambda-\left|\operatorname{det}\left(\lambda-A^{H} A\right)\right|\left(\frac{n-1}{(n+ 1) \lambda-\|A\|_{F}^{2}}\right)^{n-1}\right)^{1/2}
	$$
	Let $\lambda\to {\sigma_1^2}^+$($\lambda$ tend to $\sigma_1^2$ from the right). We get that the above equation is also true for $\lambda=\sigma_1^2$. So, for $\lambda\geqslant\sigma_1^2$, we have
	\begin{equation}\label{eq2}
		\sigma_1\leqslant\left(\lambda-\left|\operatorname{det}\left(\lambda-A^{H} A\right)\right|\left(\frac{n-1}{(n+ 1) \lambda-\|A\|_{F}^{2}}\right)^{n-1}\right)^{1/2}
	\end{equation}
	We use  induction on $k$ to prove that:
	\[
	\sigma_1\leqslant a_{k+1}\leqslant a_k
	\]
	For $k=1$, $\sigma_1\leqslant a_1$ can be obtained from the condition. In the equation (\ref{eq2}), taking $\lambda=a_1^2$, we can get
	\[
	\sigma_1\leqslant a_2=\left(a_1^2-\left|\operatorname{det}\left(a_1^2-A^{H} A\right)\right|\left(\frac{n-1} {(n+1) a_1^2-\|A\|_{F}^{2}}\right)^{n-1}\right)^{1/2}\leqslant a_1
	\]
	Suppose our conclusion holds for $k=m$, that is, $\sigma_{1} \leqslant a_{m+1} \leqslant a_{m}$. Now let's consider the case of $k=m+1$. In equation (\ref{eq2}), let $\lambda=a_{m+1}^2$, we can get
	\[
	\sigma_1\leqslant a_{m+2}=\left(a_{m+1}^2-\left|\operatorname{det}\left(a_{m+1}^2-A^{H} A\right)\right|\left(\frac{n-1}{(n+1) a_{m+1}^2-\|A\|_{F}^{2}}\right)^{n -1}\right)^{1/2}\leqslant a_{m+1}
	\]
	So $\sigma_1\leqslant a_{k+1}\leqslant a_{k}$. This proves that $\sigma_1\leqslant a_{k+1}\leqslant a_{k}(k=1,2,\cdots)$. From the monotone convergence theorem, $\lim_{k\to \infty}a_k$ exists. Let $\lim_{k\to \infty}a_k=\sigma$, then
	\[
	\sigma=\left(\sigma^2-|\det(\sigma^2 I_n-A^HA)|\left(\frac{n-1}{(n+1) \sigma^2-\|A \|_F^2}\right)^{n-1}\right)^{1/2}
	\]
	We have
	\[
	|\det(\sigma^2 I_n-A^HA)|\left(\frac{n-1}{(n+1) \sigma^2-\|A\|_F^2}\right)^{ n-1}=0.
	\]
	because 
	\[
	\frac{n-1}{(n+1) \sigma^2-\|A\|_F^2}\not=0,
	\]
	we can get
	\[
	\det(\sigma^2 I_n-A^HA)=0.
	\]
	So $\sigma^2$ is the eigenvalue of $A^HA$. Because $\sigma_1^2$ is the largest eigenvalue of $A^HA$, there is $\sigma^2\leqslant \sigma_1^2$. According to the definition of $\sigma$, we have $\sigma\geqslant \sigma_1$, so $\sigma^2\geqslant\sigma_1^2$. We get $\sigma=\sigma_1$, so $\lim_{k\to +\infty}a_k=\sigma_1$.
\end{proof}
The largest singular value has an obvious upper bound. Because $\sigma_{1}^2\leqslant\sigma_{1}^2+\cdots+\sigma_n^2=\|A\|_F^2$, so $\sigma_{1}<\|A\|_F $. We give an iterative formula for the largest singular value:
\[
a_{k+1}=\left(a_k^2-|\det(a_k^2 I_n-A^HA)|\left(\frac{n-1}{(n+1) a_k^2-\| A\|_F^2}\right)^{n-1}\right)^{1/2}, k=1,2,\cdots,
\]
where $a_1=\|A\|_F$, which converges decreasingly to $\sigma_1$.

\bibliographystyle{plain} 
\bibliography{refs}
	
\end{document}